\documentclass[12pt]{amsart}

\usepackage{fullpage}
\usepackage{mathtools}
\usepackage{amssymb,amsmath,amsthm,amscd,mathrsfs,graphicx}
\usepackage[dvipsnames]{xcolor}
\usepackage{bm}
\usepackage{hyperref}
\usepackage{dsfont}
\usepackage{enumerate}
\usepackage{epsfig}
\usepackage{float, graphicx}
\usepackage{latexsym, amsxtra}
\usepackage{pdfpages}
\usepackage{multicol}
\usepackage[normalem]{ulem}
\usepackage{psfrag}
\usepackage{stmaryrd}
\usepackage{tikz}
\usepackage[T1]{fontenc}
\usepackage{url}
\usepackage{verbatim}
\usepackage{xy}
\usepackage{indentfirst}
\usepackage{tikz-cd}
\usepackage{url}

\makeatletter
\def\thm@space@setup{%
  \thm@preskip=2ex \thm@postskip=2ex
}
\makeatother

\hypersetup{hidelinks}

\numberwithin{equation}{section}
\theoremstyle{plain}

\newtheorem{thm}{Theorem~}[section] 
\newtheorem{lem}[thm]{Lemma~}

\newtheorem{prop}[thm]{Proposition~}

\theoremstyle{remark}
\newtheorem{rmk}[thm]{Remark~}

\theoremstyle{definition}
\newtheorem{defn}[thm]{Definition~}

\newcommand{\CC}{\mathbb{C}}
\newcommand{\ZZ}{\mathbb{Z}}

\newcommand{\PP}{\mathbb{P}}

\newcommand{\NN}{\mathbb{N}}

\newcommand{\calO}{\mathcal{O}}

\newcommand\PGL{\mathrm{PGL}}

\newcommand\diag{\mathrm{diag}}
\newcommand\GL{\mathrm{GL}}

\newcommand\ord{\mathrm{ord}}

\newcommand\Lin{\mathrm{Lin}}

\newcommand\lcm{\mathrm{lcm}}

\title{On Abelian Automorphism Groups of Hypersurfaces}

 \author[Z. Zheng]{Zhiwei Zheng}
\address{Max Planck Institute for Mathematics, Bonn}
\email{zhengzw11@mpim-bonn.mpg.de}
\date{}
\begin{document}
\bibliographystyle{amsalpha}

\begin{abstract} 
Given integers $d\ge 3$ and $N\ge 3$. Let $G$ be a finite abelian group acting faithfully and linearly on a smooth hypersurface of degree $d$ in the complex projective space $\PP^{N-1}$. Suppose $G\subset \PGL(N, \CC)$ can be lifted to a subgroup of $\GL(N, \CC)$. Suppose moreover that there exists an element $g$ in $G$ such that $G/\langle g\rangle$ has order coprime to $d-1$. Then all possible $G$ are determined (Theorem \ref{theorem: main}). As an application, we derive (Theorem \ref{theorem: order}) all possible orders of linear automorphisms of smooth hypersurfaces for any given $(d, N)$. In particular, we show (Proposition \ref{proposition: cubic}) that the order of an automorphism of a smooth cubic fourfold is a factor of $21,30,32,33,36$ or $48$, and each of those $6$ numbers is achieved by a unique (up to isomorphism) cubic fourfold.
\end{abstract}

\maketitle


\section{Introduction}
\label{section: introduction}
Let $d\ge 3, N\ge 3$ be integers. We consider complex homogeneous polynomials of degree $d$ in $N$ variables. Such a polynomial $F$ defines a hypersurface $V(F)$ in the complex projective space $\PP^{N-1}$. We call such a hypersurface an $(N-2)$-fold of degree $d$. By Matsumura and Monsky \cite{matsumura1963automorphisms}, any regular automorphism of $V(F)$ is linear (namely, induced by a linear transformation of the ambient projective space) when $V(F)$ is smooth and $N\ge 4, (d,N)\ne (4,4)$. Denote by $\Lin(V(F))$ the group of linear automorphisms of $V(F)$. When $V(F)$ is smooth, the group $\Lin(V(F))$ is finite.

\smallskip

One of the goals of this paper is to determine all possible orders of linear automorphisms of smooth hypersurfaces for arbitrary $(d, N)$. A previous result by Gonz\'{a}lez-Aguilera and Liendo \cite{gonzalez2013order}\footnote{However, there is a small gap in the proof of Theorem 1.3 in \cite{gonzalez2013order}. See \cite{oguiso2019quintic} (Theorem 5.1) for a corrected statement.} classifies orders which are prime-powers coprime to $d(d-1)$. Their work is based on the fact that any automorphism of the projective space with finite order can be represented by a diagonal matrix after suitable choice of homogeneous coordinate. This classification is used by Oguiso and Yu in their work \cite{oguiso2019quintic}, which classifies maximal automorphism groups of smooth quintic threefolds. More recently, Wei and Yu \cite{AutCubic3} classified maximal automorphism groups of cubic threefolds. Both \cite{oguiso2019quintic} and \cite{AutCubic3} started with determining possible automorphism groups of small size (for example, cyclic groups and $p$-groups). Then they found all possible "combinations" (with the help of GAP software) of the small groups which result in automorphism groups. These works motivate us to study abelian automorphism groups of smooth hypersurfaces in a more general setting.
\smallskip

This paper first study abelian group actions on smooth hypersurfaces by extending Gonz\'{a}lez-Aguilera and Liendo's approach. More precisely, we will use the basic fact that any abelian subgroups of $\GL(N, \CC)$ are conjugate to subgroups of $(\CC^{\times})^N$. Our main results (Theorems \ref{theorem: abelian} and \ref{theorem: main}) classify all finite abelian groups $G$ such that $G$ can be split as $G=G_1\oplus G_2$ with $G_1$ cyclic and $\gcd(|G_2|, d-1)=1$, and $G$ admits a liftable (see Definition \ref{definition: liftable}), linear and faithful action on a smooth $(N-2)$-fold of degree $d$. 

\smallskip

We sketch our approach and explain the new ideas. Take a finite abelian group $G$ and a character $\lambda\colon G\to \CC^{\times}$. A group homomorphism $G\to (\CC^{\times})^N$ is equivalent to $N$ characters $\lambda_i\colon G\to \CC^{\times}$ for $i=1,\cdots, N$. We consider the set $S$ of monomials $\prod_{i=1}^N x_i^{\alpha_i}$ with $\sum_{i=1}^N \alpha_i=d$ and $\prod_{i=1}^N \lambda_i^{\alpha_i}=\lambda$. Then each complex linear combination of elements in $S$ gives rise to a polynomial $F(x_1, \cdots, x_N)$ of degree $d$, such that the hypersurface $V(F)\subset \PP^{N-1}_{x_1, \cdots, x_N}$ admits an action of $G$. When the data $(G, \lambda, \lambda_1, \cdots, \lambda_N)$ is suitably chosen, the set $S$ can give rise to smooth polynomials (we call a polynomial smooth if it defines a smooth hypersurface). Indeed, any liftable group actions on smooth hypersurfaces arise in this way. 
\smallskip

The main novelty of this paper is a choice of a set of very special smooth polynomials, which we call the simple polynomials. For us, a simple polynomial is the sum of finitely many polynomials of special shape that we call of type $K$ and $T$ (see Section \ref{section: notation} for a precise definition). We show that if some mild conditions about $G$ are satisfied and $G$ admits a liftable, linear and faithful action on a smooth hypersurface, then there exist characters $\lambda, \lambda_1, \cdots, \lambda_N\colon G\to \CC^{\times}$ such that the corresponding set $S$ contains the monomials of certain simple polynomial. In particular, the group $G$ is isomorphic to a subgroup of the linear automorphism group of the hypersurface defined by this simple polynomial. 
\smallskip

After we obtain our main results on abelian automorphism groups of hypersurfaces, we applied Theorem \ref{theorem: main} to prove Theorem \ref{theorem: order}, in which we classify all possible orders of linear automorphisms of smooth $(N-2)$-folds of degree $d$. We apply Theorem \ref{theorem: order} for cubic fourfolds (when $(d,N)=(3,6)$) and show that (Proposition \ref{proposition: cubic}) the order of an automorphism of a smooth cubic fourfold is a factor of $21, 30, 32, 33, 36$ or $48$. Moreover, we prove that these six numbers are achieved uniquely by smooth cubic fourfolds. We expect Theorem \ref{theorem: order} to be useful in further applications.
\smallskip

This paper is organized as follows. In Section \ref{section: notation} we introduce some notations and preliminary results. Especially, we define the simple polynomials and discuss their basic properties. Section \ref{section: simple polynomial} is devoted to the study of the diagonal automorphisms of simple hypersurfaces. Section \ref{section: abelian} is the main section, where we give the precise formulation and proof of Theorem \ref{theorem: abelian}, \ref{theorem: main} and \ref{theorem: order}. We end up with a case-study on cubic fourfolds in Section \ref{section: cubic}.
\smallskip

\noindent \textbf{Acknowledgement:} I am grateful to MPIM for hospitality and excellent research atmosphere. I thank Chenglong Yu for discussion on some examples, and Radu Laza for discussion on cubic fourfolds. After the posting of the first version of the manuscript, I learned from Gonz\'{a}lez-Aguilera, Liendo and Montero \cite{gonzalez2020lift} about the liftability property for automorphism groups of smooth hypersurfaces. I thank the authors of \cite{gonzalez2020lift} for sharing an early version of their work. Finally, I thank the referee for many helpful comments.

\section{Notation and Preliminary Results}
\label{section: notation}
\subsection{Automorphism Groups of Hypersurfaces}
Given integers $d\ge 3, N\ge 3$. All varieties considered in this paper are over the complex field $\CC$.  Let $F$ be a global section of $\calO(d)$ on $\PP^{N-1}$. Then $F$ defines an $(N-2)$-fold $V(F)$ of degree $d$ in $\PP^{N-1}$. For each choice of a homogeneous coordinate system $(x_1, \cdots, x_N)$ for $\PP^{N-1}$, the monomials in $x_1, \cdots, x_N$ of degree $d$ form a basis for the vector space $H^0(\PP^{N-1}, \calO(d))$. Thus $F$ can be uniquely expressed as a homogeneous polynomial in $x_1, \cdots, x_N$ of degree $d$. We simply denote this polynomial by $F(x_1, \cdots, x_N)$.
\smallskip

We call a homogeneous polynomial $F(x_1, \cdots, x_k)$ ($k\ge 1$) smooth if the only solution for $$\frac{\partial F}{\partial x_1}=\cdots=\frac{\partial F}{\partial x_k}=0$$ in $\CC^{k}$ is $(x_1, \cdots, x_k)=(0,\cdots, 0)$. For $k\ge 3$, a smooth homogeneous polynomial in $x_1, \cdots, x_k$ defines a smooth hypersurface in $\PP^{k-1}$. 
\smallskip

For $g\in \GL(N, \CC)$ and $F(x_1, \cdots, x_N)$ a polynomial of degree $d$, we define $g(F)=F\circ g^{-1}$. This gives rise to an action of $\GL(N, \CC)$ on the set of polynomials (in $x_1, \cdots, x_N$) of degree $d$. The class $[g]$ of $g$ in $\PGL(N, \CC)$ then sends $V(F)$ to $V(g(F))$. If $V(g(F))=V(F)$ (equivalently, $g(F)=\lambda F$ for certain $\lambda\in \CC^{\times}$), then $[g]$ defines a linear automorphism of $V(F)$. We define $\Lin(V(F))$ to be the group of linear automorphisms of $V(F)$. As we have mentioned in the introduction, by Matsumura and Monsky \cite{matsumura1963automorphisms}, the group $\Lin(V(F))$ is finite when $F$ is smooth. Moreover, when $(d, N)\ne (4,4)$ and $F$ is smooth, any regular automorphism of the variety $V(F)$ is indeed linear. We will need the following definitions.

\begin{defn}
\label{definition: liftable}
Given a linear action of a group $G$ on $\PP^{N-1}$ (equivalently, given a projective representation $G\to \PGL(N, \CC)$ of the group $G$). We say this action is liftable, if there exists a group homomorphism $G\to \GL(N, \CC)$, such that its composition with the natural projection $\GL(N, \CC)\to \PGL(N, \CC)$ is the given projective representation. We call such a homomorphism $G\to \GL(N, \CC)$ a lifting for the action.
\end{defn}

\begin{defn}
Given a linear action of a group $G$ on $\PP^{N-1}$. Suppose 
\begin{enumerate}[(a)]
\item the action is liftable, and 
\item there exist a lifting $j\colon G\to \GL(N, \CC)$, a hypersurface $V(F)$ in $\PP^{N-1}$ and a character $\lambda\colon G\to \CC^{\times}$ such that $j(g)(F)=\lambda(g)F$ for any $g\in G$. 
\end{enumerate}
Then we say (the action of) $G$ is $(F, \lambda)$-liftable. 
\end{defn}

It is straightforward to check that any linear action of a cyclic group on $\PP^{N-1}$ is liftable. In our main results (Theorems \ref{theorem: abelian} and \ref{theorem: main}), liftabilities of abelian group actions are always required. See Remark \ref{remark: lift} for more discussion.

\subsection{Simple Polynomials and Simple Hypersurfaces}
For integers $d\ge 3$ and $k\ge 1$, we call the polynomial $x_1^{d-1}x_2+x_2^{d-1}x_3+\cdots+x_k^{d-1}x_1$ of type $K_k$ (here the letter $K$ stands for Klein), and the polynomial $x_1^{d-1}x_2+\cdots+x_{k-1}^{d-1}x_k+x_k^d$ of type $T_k$. A simple polynomial is the sum of finitely many polynomials of type $K$ or $T$ with independent variables and the same degree. 
\smallskip

It is straightforward to show that simple polynomials of degree at least $3$ are smooth. We call a hypersurface in the projective space a simple hypersurface, if it is defined by a simple polynomial. We write the type of a simple polynomial as direct sum of the types of each component. For a simple polynomial, if all of its components are of type $K$, then we call it $K$-pure. For example, the Fermat polynomial $x_1^d+\cdots+x_k^d$ is of type $K_1^{\oplus k}$, and is $K$-pure.

\smallskip

Let $F=F(x_1, \cdots, x_k)$ be a smooth polynomial of degree $d$. Then for each $i\in\{1,\cdots, k\}$, there exists a monomial with nonzero coefficient in $F$ such that the multiplicity of $x_i$ in this monomial is at least $d-1$ (see \cite{gonzalez2013order}, Lemma 1.2). For any monomial of degree $d$, there is at most one $x_i$ with multiplicity at least $d-1$. This leads us to consider polynomials of the form 
\begin{equation*}
F_I(x_1, \cdots, x_k)=x_1^{d-1} x_{i_1}+\cdots+x_N^{d-1} x_{i_k}
\end{equation*} 
where $I=(d, \overline{i})=(d, i_1, \cdots, i_k)\in \NN^{\ge 3}\times\{1,\cdots, k\}^k$. 

\begin{prop}
\label{proposition: smooth criterion}
Given $I=(d, i_1, \cdots, i_k)\in \NN^{\ge 3}\times\{1,\cdots, k\}^k$, the following three statements are equivalent:
\begin{enumerate}[(1)]
\item The polynomial $F_I$ is smooth.
\item We cannot find $1\le a<b\le k$ such that $i_a=i_b$ and $i_a\notin\{a,b\}$. 
\item The polynomial $F_I$ is simple.
\end{enumerate}
\end{prop}
\begin{proof}
By definition of simpleness, the statement $(3)$ implies $(1)$. We next show $(1)$ implies $(2)$. Suppose there exist $1\le a<b\le k$ such that $i_a=i_b$ and $i_a\notin\{a,b\}$. Without loss of generality, we assume that $a=1, b=2$ and $i_a=i_b=3$. We claim that $V(F_I)$ is singular at the point $p=[1: \mathrm{exp}(\frac{\pi\sqrt{-1}}{d-1}):0:\cdots:0]$. Indeed, it is straightforward to check that $\frac{\partial F}{\partial x_i}(p)=0$ for any $i=1, 2, \cdots, N$. Therefore, if $F_I$ is smooth, then $(2)$ holds.
\smallskip

Next we prove that $(2)$ implies $(3)$. Suppose we cannot find $1\le a<b\le k$ such that $i_a=i_b$ and $i_a\notin\{a,b\}$. We consider a directed graph $H$ with $1,\cdots, k$ being the vertices and $a\to i_a$ ($a=1,\cdots, N$) being the edges. Notice that each vertex of $H$ is the beginning point of exactly one edge, and each vertex is the ending point of at most two edges. Moreover, due to the assumption, if one vertex is the ending point of two edges, then one of the edges is a circle based on that vertex. These restrictions immediately imply that the components of $H$ are of types
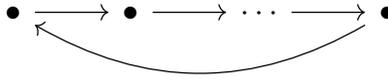
\begin{equation*}
\label{diagram: A}
\begin{tikzcd}
\bullet\arrow{r} &\bullet\arrow{r} & \cdots\arrow{r} &\bullet\arrow[bend left]{lll}
\end{tikzcd}
\end{equation*}
or
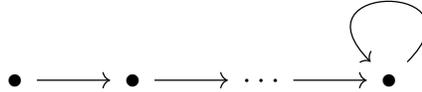
\begin{equation*}
\label{diagram: B}
\begin{tikzcd}
\bullet\arrow{r} &\bullet\arrow{r} & \cdots\arrow{r} &\bullet\arrow[loop]
\end{tikzcd}
\end{equation*}
Thus $F_I$ is a simple polynomial.
\end{proof}

\section{Automorphisms of Simple Hypersurfaces}
\label{section: simple polynomial}
In this section we are given integers $d\ge 3, N\ge 2$. For $\lambda_1, \cdots, \lambda_N\in \CC^{\times}$, we denote by $\diag[\lambda_1:\cdots:\lambda_N]$ the class of the diagonal matrix $\diag(\lambda_1, \cdots, \lambda_N)$ in $\PGL(N, \CC)$. The element $\diag[\lambda_1:\cdots:\lambda_N]$ is acting on $\PP^{N-1}_{x_1, \cdots, x_N}$, sending $[x_1:\cdots:x_N]$ to $[\lambda_1 x_1:\cdots:\lambda_N x_N]$. We call this a diagonal automorphism. For a smooth polynomial $F(x_1, \cdots, x_N)$, we define $G_F$ to be the group of diagonal automorphisms $\diag [\lambda_1: \cdots : \lambda_N]$ such that
\begin{equation*}
F(\lambda_1 x_1, \cdots, \lambda_N x_N)=\lambda F(x_1, \cdots, x_N)
\end{equation*}
for certain $\lambda\ne 0$. The group $G_F$ is a subgroup of $\Lin(V(F))$, hence a finite abelian group. 

\smallskip

Next we characterize the groups $G_F$ for any simple polynomials $F$. We first consider polynomials of type $K$ or type $T$.
\begin{lem}
\label{lemma: klein}
Let $F(x_1, \cdots, x_N)=x_1^{d-1}x_2+\cdots+x_N^{d-1}x_1$ be the polynomial of degree $d$ and type $K_N$. Then $G_F$ is a cyclic group of order $\frac{|1-(1-d)^N|}{d}$ with a generator given by $\diag [\zeta:\zeta^{1-d}:\zeta^{(1-d)^2}:\cdots:\zeta^{(1-d)^{N-1}}]$ where $\zeta$ is a primitive $|1-(1-d)^N|$-root of unity.
\end{lem}
\begin{proof}
The element $g=\diag [\zeta:\zeta^{1-d}:\zeta^{(1-d)^2}:\cdots:\zeta^{(1-d)^{N-1}}]$ is a linear automorphism of $V(F)$ of order $\frac{|1-(1-d)^N|}{d}$. Denote by $\langle g\rangle$ the cyclic group generated by $g$. Then we have a group inclusion $\langle g\rangle\hookrightarrow G_F$, which is an isomorphism as we are going to show.
\smallskip

Take $h\in G_F$ and choose one of its preimages $\widetilde{h}=\diag(\lambda_1, \cdots, \lambda_N)$ in $\GL(N, \CC)$ such that $\widetilde{h} (F)=F$. From $\widetilde{h} (x_1^{d-1}x_2)=x_1^{d-1}x_2$ we have $\lambda_1^{d-1}\lambda_2=1$, hence $\lambda_2=\lambda_1^{1-d}$. Accordingly, we have $\lambda_i=\lambda_1^{(1-d)^{i-1}}$ for any $i=1, \cdots, N$. From $\lambda_N^{d-1}\lambda_1=1$ we have $\lambda_1^{1-(1-d)^N}=1$, namely $\lambda_1$ is a $|1-(1-d)^N|$-root of unity. Thus $h\in \langle g\rangle$. We conclude that $G_F=\langle g\rangle$.
\end{proof}

\begin{lem}
\label{lemma: type T}
Let $F(x_1, \cdots, x_N)=x_1^{d-1}x_2+\cdots+x_{N-1}^{d-1}x_N+x_N^d$ be the polynomial of  degree $d$ and type $T_N$. Then $G_F$ is a cyclic group of order $(d-1)^{N-1}$ with a generator given by $\diag [\zeta:\zeta^{1-d}:\zeta^{(1-d)^2}:\cdots:\zeta^{(1-d)^{N-1}}]$, where $\zeta$ is a primitive $(d-1)^{N-1}$-root of unity.
\end{lem}
\begin{proof}
The element $g=\diag [\zeta:\zeta^{1-d}:\zeta^{(1-d)^2}:\cdots:\zeta^{(1-d)^{N-1}}]$ is a linear automorphism of $V(F)$ of order $(d-1)^{N-1}$. We have the group inclusion $\langle g\rangle\hookrightarrow G_F$. Same proof as Lemma \ref{lemma: klein} implies that $G_F=\langle g\rangle$.
\end{proof}

\begin{prop}
\label{proposition: simple}
Let $F(x_1, \cdots, x_N)$ be a simple polynomial of degree $d$ and type $K_{a_1}\oplus\cdots\oplus K_{a_t}\oplus T_{b_1}\oplus\cdots\oplus T_{b_s}$, with $t, s$ being non-negative integers. Then there is an isomorphism of finite abelian groups:
\begin{equation}
\label{equation: GF}
G_F\cong (\prod_{i=1}^t \ZZ/(|1-(1-d)^{a_i}|)\ZZ\times \prod_{j=1}^s \ZZ/(d(d-1)^{b_j-1})\ZZ)/\langle (u_1, \cdots, u_t, v_1, \cdots, v_s)\rangle.
\end{equation}
Here $u_i, v_j$ are elements in the corresponding summands with orders equal to $d$.

\end{prop}
\begin{proof}
Let $F=F_1+\cdots+F_{t+s}$ where $F_i$'s have independent variables, $F_i$ ($i=1, \cdots, t$) are of type $K_{a_i}$ and $F_{t+j}$ ($j=1, \cdots, s$) are of type $T_{b_j}$. For convenience, we denote the $N$ variables $x_1, \cdots, x_N$ in a new way as follows. Let the variables appearing in $F_i$ ($i=1, \cdots, t$) be $x_{i,1}, \cdots, x_{i, a_i}$, and the variables appearing in $F_{t+j}$ ($j=1, \cdots, s$) be $x_{t+j,1}, \cdots, x_{t+j, b_j}$.

\smallskip

Now we define $t+s$ elements $g_1, \cdots, g_t, h_1, \cdots, h_s\in \GL(N, \CC)$ as follows. For $i=1, \cdots, t$, we take $\zeta_i=\mathrm{exp}(\frac{2\pi\sqrt{-1}}{|1-(1-d)^{a_i}|})$. Let $g_i(x_{i,k})=\zeta_i^{(1-d)^{k-1}}x_{i,k}$ for $k=1, \cdots, a_i$, and $g_i(x_{i',k})=x_{i',k}$ for $i'\ne i$. Then $g_i F=F$ and $g_i$ generate a subgroup isomorphic to $\ZZ/(|1-(1-d)^{a_i}|)\ZZ$ in $\GL(N, \CC)$. For $j=1, \cdots, s$, we take $\xi_j=\mathrm{exp}(\frac{2\pi\sqrt{-1}}{d(d-1)^{b_j-1}})$. Let $h_j(x_{t+j, k})=\xi_j^{(1-d)^{k-1}}x_{t+j, k}$ for $k=1, \cdots, b_j$, and $h_j(x_{i',k})=x_{i',k}$ for $i'\ne t+j$. Then $h_j F=F$ and $h_j$ generate a subgroup isomorphic to $\ZZ/(d(d-1)^{b_j-1})\ZZ$ in $\GL(N, \CC)$. The $t+s$ elements $g_1, \cdots, g_t, h_1, \cdots, h_s$ generate a subgroup $\widetilde{G}\subset \GL(N, \CC)$ which is isomorphic to $\prod_{i=1}^t \ZZ/(|1-(1-d)^{a_i}|)\ZZ\times \prod_{j=1}^s \ZZ/(d(d-1)^{b_j-1})\ZZ$. Since the group $\widetilde{G}$ leaves $F$ invariant, we have a group homomorphism $\Pi\colon\widetilde{G}\longrightarrow G_F$. We next describe the kernel of $\Pi$ and show that $\Pi$ is surjective.

\smallskip

Let $\varphi$ be an element in the kernel of $\Pi$. Then $\varphi=\diag(\rho, \cdots, \rho)$ for certain $\rho \in \CC^{\times}$. From the definition of $\widetilde{G}$, we can write $\varphi=u_1\cdots u_t v_1\cdots v_s$ for $u_i=g_i^{\alpha_i}$, and $v_j=h_j^{\beta_j}$. Since $u_i$ acts with the same scalar on $x_{i,k}$ for $k=1, \cdots, a_i$, we must have $\frac{|1-(1-d)^{a_i}|}{d}\mid \alpha_i$, and the order of $u_i$ is either $d$ or $1$. Similarly, we  have $(d-1)^{b_j-1}\mid \beta_j$ and the order of $v_j$ is either $d$ or $1$. Thus $\rho$ is a $d$-root of unity. The kernel of $\Pi$ is generated by the element 
\begin{equation*}
\diag(\mathrm{exp}(\frac{2\pi\sqrt{-1}}{d}), \cdots, \mathrm{exp}(\frac{2\pi\sqrt{-1}}{d}))=u_1\cdots u_t v_1\cdots v_s
\end{equation*}
with $u_i=g_i^{\alpha_i}, v_j=h_j^{\beta_j}$, $\alpha_i=\frac{|1-(1-d)^{a_i}|}{d}$ and $\beta_j=(d-1)^{b_j-1}$.

\smallskip

Take $f\in G_F$. The restriction $f_i$ of $f$ to $\PP_{x_{i,1}, \cdots, x_{i, a_i}}^{a_i-1}$ is a linear automorphism of $V(F_i)$, namely, $f_i\in G_{F_i}$. By Lemma \ref{lemma: klein}, the restriction of $[g_i]$ to $\PP_{x_{i,1}, \cdots, x_{i, a_i}}^{a_i-1}$ generates $G_{F_i}$. By Lemma \ref{lemma: type T}, the restriction of $[h_j^d]$ to $\PP_{x_{t+j,1}, \cdots, x_{t+j, b_j}}^{b_i-1}$ generates $G_{F_{t+j}}$. Therefore, we can choose an element $f_1\in \widetilde{G}$ such that the restrictions of $f [f_1]^{-1}$ to  $\PP_{x_{i,1}, \cdots, x_{i, a_i}}^{a_i-1}$ and $\PP_{x_{t+j,1}, \cdots, x_{t+j, b_j}}^{b_i-1}$ are trivial for all $i=1, \cdots, t$ and $j=1, \cdots, s$. Take a preimage $\widetilde{f}\in \GL(N, \CC)$ of $f [f_1]^{-1}$ such that $\widetilde{f} F=F$. Then $\widetilde{f}$ is diagonal with all eigenvalues being $d$-roots of unity. Moreover, $\widetilde{f}$ acts with the same scalar on $x_{i,k}$ for fixed $i$. Notice that $g_i^{\frac{|1-(1-d)^{a_i}|}{d}}$ acts with scalar $\mathrm{exp}(\frac{2\pi\sqrt{-1}}{d})$ on $x_{i,1}, \cdots, x_{i, a_i}$, for $i=1, \cdots, t$. And $h_j^{(d-1)^{b_j-1}}$ acts with scalar $\mathrm{exp}(\frac{2\pi\sqrt{-1}}{d})$ on $x_{t+j, 1}, \cdots, x_{t+j, b_j}$, for $j=1, \cdots, s$. Therefore, the element $\widetilde{f}$ is generated by $g_1, \cdots, g_t, h_1, \cdots, h_s$ in $\GL(N, \CC)$. Thus $\widetilde{f}\in \widetilde{G}$. This implies that $f=\Pi(\widetilde{f}f_1)$ lies in the image of $\Pi$. We conclude that $\Pi$ is surjective. We then have equation \eqref{equation: GF} and the proposition follows.
\end{proof}

We end up this section with a lemma which will be used (together with Proposition \ref{proposition: simple}) to prove Theorem \ref{theorem: order}.
\begin{lem}
\label{lemma: order}
Given integers $k\ge 2, d\ge 2$ and an abelian group $G=G_1\oplus\cdots\oplus G_k/\langle(u_1, \cdots, u_k)\rangle$ with $G_1, \cdots, G_k$ cyclic and for each $i=1, \cdots, k$, $u_i\in G_i$ is an element of order $d$. Then an integer $n$ is the order of an element in $G$ if and only if $n$ is a factor of $\lcm(|G_1|, \cdots, |G_k|)$.
\end{lem}
\begin{proof}
Denote $L=\lcm(|G_1|, \cdots, |G_k|)$. An integer $n$ is the order of an element in $G_1\oplus\cdots\oplus G_k$ if and only if $n$ is a factor of $L$. Thus the orders of any elements in $G$ are factors of $L$. To prove the lemma, it suffices to construct an element in $G$ with order equals to $L$. Let $p_1, \cdots, p_t$ be all different prime factors of $d$. For $i\in\{1, \cdots, t\}$, take $j_i\in \{1, \cdots, k\}$ such that the exponent of $p_i$ in $|G_{j_i}|$ is minimal among the exponents of $p_i$ in $|G_1|, \cdots, |G_k|$. For $j\in \{1, \cdots, k\}$, take an element $g_j\in G_j$ with order $|G_j|/\prod_{j_i=j} p_i$. Suppose the order of $[(g_1, \cdots, g_k)]\in G_1\oplus\cdots\oplus G_k/\langle(u_1, \cdots, u_k)\rangle$ is $a$. We next show that $a=L$. 

\smallskip

There exists a positive integer $b$ such that $(g_1^a, \cdots, g_k^a)=(u_1^b, \cdots, u_k^b)$. Suppose to the contrary that $(g_1^a, \cdots, g_k^a)$ is not the identity element. Since $u_1, \cdots, u_k$ have the same order $d$, the powers $u_1^b, \cdots, u_k^b$ also have the same order $d'$. Here $d'\ne 1$ and $d'\mid d$. Let $p$ be a prime factor of $d'$. Without loss of generality, we assume that $p=p_1$ and $j_1=1$. Denote by $\alpha_1, \alpha_2$ the exponents of $p$ in $|G_1|, |G_2|$ respectively (we can do this because $k\ge 2$). From $j_1=1$ we have $\alpha_1\le \alpha_2$. By constructions of $g_1, g_2$, we know that the exponent of $p$ in $\ord(g_1)$ is less than $\alpha_1$, and the exponent of $p$ in $\ord(g_2)$ is equal to $\alpha_2$. Therefore, the exponent of $p$ in $\ord(g_1)$ is less than that in $\ord(g_2)$. On the other hand, the orders of $g_1^a=u_1^b$ and $g_2^a=u_2^b$ all equal to $d'$ and $p\mid d'$, which is a contradiction. 

We conclude that $(g_1^a, \cdots, g_k^a)$ is the identity element in $G_1\oplus\cdots\oplus G_k$. Notice that the order of $(g_1, \cdots, g_k)$ in $G_1\oplus\cdots\oplus G_k$ is equal to $\lcm(\ord(g_1), \cdots, \ord(g_k))$, which is equal to $L=\lcm(|G_1|, \cdots, |G_k|)$ from the constructions of $g_1, \cdots, g_k$. Thus $L|a$, which implies that $L=a$. The lemma then follows.
\end{proof}

\section{Abelian Automorphism Groups}
\label{section: abelian}
In this section we formulate and prove our main theorems. For a homogeneous polynomial $F(x_1, \cdots, x_N)$ of degree $d$, we denote by $S(F)$ the set of monomials $\prod x_i^{\alpha_i}$ with nonzero coefficients in $F$.

The next lemma is equivalent to Lemma 3.2 in \cite{oguiso2019quintic}. The formulation here is easier to use in our setting.
\begin{lem}
\label{lemma: criterion}
Suppose $F(x_1, \cdots, x_N)$ is a smooth polynomial of degree $d$. Then for non-negative integers $a>b$, and distinct integers $i_1, \cdots, i_a, j_1, \cdots, j_b\in\{1, \cdots, N\}$, there exists a monomial $m\in S(F)$, such that the sum of the degrees of $x_{i_1}, \cdots, x_{i_a}$ in $m$ is at least $d-1$, and the variants $x_{j_1}, \cdots, x_{j_b}$ do not appear in $m$.
\end{lem}
\begin{proof}
Suppose there does not exist such a monomial. Let $H$ be the part of $F$ in which the monomials have total degree at least $d-1$ in $x_{i_1}, \cdots, x_{i_a}$. Then each monomial in $H$ has exactly total degree $d-1$ in $x_{i_1}, \cdots, x_{i_a}$, and total degree $1$ in $x_{j_1}, \cdots, x_{j_b}$. Thus we can write 
\begin{equation*}
H=F_1(x_{i_1}, \cdots, x_{i_a})x_{j_1}+\cdots+F_b(x_{i_1}, \cdots, x_{i_a})x_{j_b}.
\end{equation*}
Since $b<a$, there exists a nonzero solution $(x_{i_1}=t_1, \cdots, x_{i_a}=t_a)$ for equations $F_1=\cdots=F_b=0$. By straightforward calculation, the polynomial $F$ is singular at the point $[x_1:\cdots:x_N]$ with $x_{i_l}=t_l$ and other coordinates vanish. This contradicts the assumption that $F$ is smooth.
\end{proof}

Our first main theorem is:
\begin{thm}
\label{theorem: abelian}
Given integers $d\ge 3$ and $N\ge 3$. Let $G$ be a finite abelian group acting linearly and faithfully on a smooth $(N-2)$-fold in $\PP^{N-1}$ of degree $d$, such that $\gcd(|G|, d-1)=1$ and the action is liftable. Then we can choose a coordinate system $(x_1, \cdots, x_N)$ for $\PP^{N-1}$, such that there exists a $K$-pure polynomial $F(x_1, \cdots, x_N)$ of degree $d$ with $G<G_F$.
\end{thm}
\begin{proof}
By the assumption, we can choose a homogeneous coordinate system $(x_1, \cdots, x_N)$ for $\PP^{N-1}$, such that $G$ admits a linear, faithful, liftable and diagonal action on a smooth hypersurface $V(F)\subset \PP^{N-1}_{x_1, \cdots, x_N}$ of degree $d$. We can take a lift $G\subset\GL(N, \CC)$ and a character $\lambda\colon G\to \CC^{\times}$, such that $g(F)=\lambda(g)F$ for any $g\in G$. For each $a=1, \cdots, N$, let $\lambda_a$ be the character of $G$ such that $g(x_a)=\lambda_a(g) x_a$ for any $g\in G$. Since $F(x_1, \cdots, x_N)$ is smooth, it contains for each $i\in\{1,\cdots, N\}$ at least one monomial with nonzero coefficient in which the multiplicity of $x_i$ is at least $d-1$. Next we show the existence of $\overline{i}=(i_1, \cdots, i_N)\in \{1, \cdots, N\}^N$ such that $\lambda_a^{d-1}\lambda_{i_a}=\lambda$ for any $a\in\{1, \cdots, N\}$, and $F_{d,\overline{i}}$ is $K$-pure.
\smallskip

Since $|G|$ is coprime to $d-1$, two characters $\lambda_a$ and $\lambda_b$ are different if and only if $\lambda_a^{d-1}$ and $\lambda_b^{d-1}$ are different. Suppose first that $\lambda_1$ is different from the other $N-1$ characters $\lambda_2, \cdots, \lambda_N$, then we simply choose $i_1\in\{1,\cdots,N\}$ with $x_1^{d-1}x_{i_1}\in S(F)$ in order to obtain the result. Secondly, we consider the case where there exist other characters equal to $\lambda_1$. Without loss of generality, we label them by $\lambda_2, \cdots, \lambda_a$. If $\lambda=\lambda_1^d$, then $x_1^d, \cdots, x_a^d$ are $(G, \lambda)$-invariant. We can take $i_1=1,\cdots, i_a=a$. Suppose we have $\lambda\ne\lambda_1^d$. Denote $\lambda'=\lambda/\lambda_1^{d-1}$. Let $x_{a+1}, \cdots, x_{a+b}$ be coordinates with associated characters equal to $\lambda'$.

Suppose $b<a$. By Lemma \ref{lemma: criterion}, there exists a monomial $m$ such that the total degree of $x_1, \cdots, x_a$ in $m$ is at least $d-1$, and the variables $x_{a+1}, \cdots, x_{a+b}$ do not appear in $m$. This implies the existence of an index $c\in \{a+b+1, \cdots, N\}$ such that $\lambda_c=\lambda'$, which is a contradiction. Thus we have $b\ge a$. This allows us to choose $i_1=a+1, i_2=a+2, \cdots, i_a=2a$.
\smallskip

We can continue the above argument for the rest coordinates, and obtain $\overline{i}\in \{1, \cdots, N\}^N$. From the construction, we have that for different $a, b\in\{1, \cdots, N\}$, the numbers $i_a$ and $i_b$ are different. Therefore, the polynomial $F=F_{d,\overline{i}}$ is $K$-pure with $G<G_F$.
\end{proof}

We hope to weaken the condition $\gcd(|G|, d-1)=1$ in Theorem \ref{theorem: abelian}. In other words, we want to control the structures of finite abelian automorphism groups $G$ of smooth hypersurfaces when $|G|$ may have factors not coprime to $d-1$. An attempt in this direction leads to the next theorem.  

\begin{thm}
\label{theorem: main}
Given integers $d\ge 3$ and $N\ge 3$. Let $G$ be a finite abelian group with a linear, faithful and liftable action on a smooth $(N-2)$-fold of degree $d$. Assume that there exists a decomposition $G=G_1\oplus G_2$ such that $G_1$ is cyclic and $G_2$ has order coprime to $d-1$. Then there exists a simple polynomial $F_0(x_1, \cdots, x_{N_0})$ (for $N_0\le N$) of degree $d$, such that $F_0$ has at most one component of type $T_{\ge 2}$, and $G$ is isomorphic to a subgroup of the group $G_{F_0}$ of diagonal automorphisms of $V(F_0)$. If $\gcd(|G|, d-1)\ne 1$, then $F_0$ can be chosen with exactly one component of type $T_{\ge 2}$. 
\end{thm}

We need the following lemma for the proof of Theorem \ref{theorem: main}.
\begin{lem}
\label{lemma: lifting}
Given integers $d\ge 3$ and $N\ge 3$. Let $G$ be a finite abelian group acting faithfully and linearly on an $(N-2)$-fold $V(F)$ of degree $d$. Suppose the action is liftable. Then there exists a character $\lambda$ of $G$ such that $G$ is $(F, \lambda)$-liftable, and for any $g\in G$ with order coprime to $d$, we have $\lambda(g)=1$.
\end{lem}
\begin{proof}
Take a decomposition $G=G_1\oplus\cdots\oplus G_k$ such that each factor $G_i$ is cyclic with order a prime power $p_i^{\alpha_i}$. For each $i=1, \cdots, k$, take $g_i$ be a generator of $G_i$, and take a preimage $\widetilde{g_i}$ of $g_i$ in $\GL(N, \CC)$, such that the order of $\widetilde{g_i}$ is equal to $p_i^{\alpha_i}$. For each $i\in\{1, \cdots, k\}$ with $p_i\nmid d$, we modify $\widetilde{g_i}$ as follows. Let $a$ be the complex number such that $\widetilde{g_i}F=aF$. Since $\widetilde{g_i}$ has order $p_i^{\alpha_i}$, the number $a$ is a $p_i^{\alpha_i}$-root of unity. Since $(p_i^{\alpha_i}, d)=1$, we can take another $p_i^{\alpha_i}$-root of unity, denoted by $b$, such that $a=b^d$. Replace $\widetilde{g_i}$ by $b \widetilde{g_i}$, we have that $\widetilde{g_i}(F)=F$.
\smallskip

Now the subgroup $\widetilde{G}$ of $\GL(N, \CC)$ generated by $\widetilde{g_1}, \cdots, \widetilde{g_k}$ is a lifting of $G$. Take the character $\lambda$ of $G$ with $\widetilde{g_i} F=\lambda(g_i)F$. Then $(F, \lambda)$ satisfies the requirement in the lemma.
\end{proof}

We are now ready to prove Theorem \ref{theorem: main}.

\begin{proof}[Proof of Theorem \ref{theorem: main}]
If $\gcd(|G|, d-1)=1$, the theorem follows from Theorem \ref{theorem: abelian}. We next assume $\gcd(|G|, d-1)\ne 1$. We can take a decomposition $G=G_1\oplus G_2$, where $G_1=\langle g_0\rangle$ is cyclic with all prime factors of $\ord(g_0)$ also factors of $d-1$, and $G_2$ has order coprime to $d-1$. 

\smallskip

Suppose $G$ admits a linear, faithful, liftable and diagonal action on $V(F)$ for $F=F(x_1, \cdots, x_N)$ smooth of degree $d$. Since the action is liftable, we can take a lifting $\widetilde{G}\subset \GL(N, \CC)$. For element $g\in G$ we denote by $\widetilde{g}$ the preimage of $g$ in $\widetilde{G}$. Let $\lambda$ be the character of $G$ such that $\widetilde{g}(F)=\lambda(g)F$ for each $g\in G$. By Lemma \ref{lemma: lifting} we can make choice of $\widetilde{G}$ such that $\lambda(g)=1$ when $\gcd(\ord(g), d)=1$. Since $\ord(g_0)$ is coprime to $d$, we have $\lambda(g_0)=1$.  For each $k\in\{1, \cdots, N\}$, let $\lambda_k$ be the character of $G$ such that $\widetilde{g} (x_k)=\lambda_k(g) x_k$ for any $g\in G$.

\smallskip

Let $S$ be the set of monomials $\prod_{i=1}^{N} x_i^{\alpha_i}$ such that $\prod_{i=1}^N \lambda_i^{\alpha_i}=\lambda$. Then $S(F)\subset S$. From the smoothness of $F$, we have that for any $i\in\{1, \cdots, N\}$, there exists $j$, such that $x_i^{d-1}x_j\in S(F)\subset S$.

\smallskip

Let $l$ be the minimal integer such that $\ord(g_0)\mid (d-1)^l$. Since $\ord(g_0)\ne 1$, $l$ is at least $1$. There exists a prime factor $p$ of $d-1$, such that the exponent of $p$ in $\ord(g_0)$ is greater than that in $(d-1)^{l-1}$. Denote by $\alpha$ the exponent of $p$ in $\ord(g_0)$. Let $g_1=g_0^{\frac{\ord(g)}{p^{\alpha}}}$. Then $g_1\in G_1$ has order $p^{\alpha}$.
\smallskip 

The values $\lambda_1(g_1), \cdots, \lambda_N(g_1)$ are $p^{\alpha}$-roots of unity and (by the faithfulness of the action of $G$) at least one of them is primitive. Take $i_1\in \{1,\cdots, N\}$ such that $\lambda_{i_1}(g_1)=\zeta$ is a primitive $p^{\alpha}$-root of unity.  We have $\widetilde{g_1}(x_{i_1}^d)=\zeta^d x_{i_1}^d$ and $\zeta^d\ne 1$. On the other hand, we have $\lambda(g_1)=1$. Thus $x_{i_1}^d\notin S$. 

\smallskip

We can find distinct elements $i_1, i_2, \cdots, i_k\in \{1,\cdots, N\}$, such that $x_{i_t}^{d-1} x_{i_{t+1}}\in S$ for $t=1,\cdots, k-1$, and there exists $l'\in \{0,\cdots, k-1\}$ such that $x_{i_k}^{d-1} x_{i_{l'+1}}\in S$. We ask $k$ to be minimal among all possible choices. Then $\lambda_{i_1}, \cdots, \lambda_{i_k}$ are distinct. Without loss of generality, we assume $i_t=t$ for $t=1,\cdots, k$. 
\smallskip

For $t\in \{1,2,\cdots, k-1\}$, since $x_t^{d-1} x_{t+1}\in S$, we have $\lambda_t(g_1)^{d-1}\lambda_{t+1}(g_1)=\lambda(g_1)=1$. This implies that $\lambda_{t+1}(g_1)=\lambda_t(g_1)^{1-d}$. Then we have $\lambda_i(g_1)=\zeta^{(1-d)^{i-1}}$ for $i=1, 2, \cdots, k$. 
\smallskip

Since $x_k^{d-1}x_{l'+1}\in S$, we have $\lambda_{l'+1}(g_1)=\lambda_k(g_1)^{1-d}=\lambda_{l'+1}(g_1)^{(1-d)^{k-l'}}$. Notice that $\lambda_{l'+1}(g_1)$ is a $p^{\alpha}$-root of unity. Thus $\lambda_{l'+1}(g_1)=1$. This implies that $l'\ge l$ and $\lambda_t(g_1)=1$ for $t=l'+1, \cdots, k$. 
\smallskip

Next we show that $l'=l$. Assume to the contrary that $l'>l$. For any integer $t\in \{l+1,\cdots, k\}$, we have $\lambda_{t}(g_0)=\lambda_1(g_0)^{(1-d)^{t-1}}$. Since $\ord(g_0)\mid (d-1)^l$ and $t\ge l+1$, we conclude that $\lambda_t(g_0)=1$. In particular, $\lambda_{l'}(g_0)=\lambda_{k}(g_0)=1$. From $\lambda_{l'}^{d-1}\lambda_{l'+1}=\lambda_{k}^{d-1}\lambda_{l'+1}=\lambda$ we obtain $\lambda_{l'}^{d-1}=\lambda_k^{d-1}$. Since $\gcd(|G_2|, d-1)=1$, the restrictions of $\lambda_{l'}$ and $\lambda_k$ to $G_2$ are the same. We conclude that $\lambda_{l'}=\lambda_k$. This contradicts the minimality of $k$. Therefore, we have $l'=l$.

\smallskip

Consider the restrictions of $\lambda_1, \cdots, \lambda_k$ to $G_2$. From $\lambda_l^{d-1} \lambda_{l+1}=\lambda_k^{d-1}\lambda_{l+1}$ we obtain $\lambda_l\big{|}_{G_2}=\lambda_k\big{|}_{G_2}$. Then by $\lambda_{l-1}^{d-1}\lambda_l=\lambda_{k-1}^{d-1}\lambda_k$ we obtain $\lambda_{l-1}\big{|}_{G_2}=\lambda_{k-1}\big{|}_{G_2}$. Continuing this process, we obtain $\lambda_t\big{|}_{G_2}=\lambda_{t+k-l}\big{|}_{G_2}$ for $t\in\{1,\cdots, l\}$. In particular, each $\lambda_t\big{|}_{G_2}$ ($t=1,\cdots, k$) is equal to one of $\lambda_{l+1}\big{|}_{G_2},\cdots, \lambda_k\big{|}_{G_2}$.
\smallskip

After applying permutations to the coordinates $x_{k+1}, \cdots, x_{N}$, we can take integers $k=k_0<k_1<\cdots<k_s=N_1\le N$, such that for each $i=1, \cdots, s$, the Klein polynomial
\begin{equation*}
F_i(x_{k_{i-1}+1}, \cdots, x_{k_i})=x_{k_{i-1}+1}^{d-1}x_{k_{i-1}+2}+\cdots+x_{k_i-1}^{d-1}x_{k_i}+x_{k_i}^{d-1}x_{k_{i-1}+1}
\end{equation*}
satisfies $S(F_i)\subset S(F)$. Moreover, for any $t\in\{k+1, \cdots, N_1\}$, the character $\lambda_t\big{|}_{G_2}$ is different from $\lambda_{l+1}\big{|}_{G_2}, \cdots, \lambda_k\big{|}_{G_2}$. We ask $N_1$ to be maximal among all possible choices. We next prove that the action of $G_2$ on $\PP_{x_{l+1}, \cdots, x_{N_1}}^{N_1-l-1}$ is faithful. Suppose not, then there exists $i_1\in\{N_1+1, \cdots, N\}$ such that $\lambda_{i_1}\big{|}_{G_2}$ is different from $\lambda_{l+1}\big{|}_{G_2}, \cdots, \lambda_{N_1}\big{|}_{G_2}$. 
\smallskip

Since $F$ is smooth, there exists $i_2\in \{1, \cdots, N\}$ such that $x_{i_1}^{d-1} x_{i_2}\in S(F)$. Suppose there exists $j\in\{l+1, \cdots, N_1\}$ such that $\lambda_{i_2}\big{|}_{G_2}=\lambda_j\big{|}_{G_2}$. Then $(\lambda_{i_1}^{d-1}\lambda_{i_2})\big{|}_{G_2}=\lambda\big{|}_{G_2}=(\lambda_{j-1}^{d-1}\lambda_j)\big{|}_{G_2}$. This implies that $\lambda_{i_1}\big{|}_{G_2}=\lambda_{j-1}\big{|}_{G_2}$. This contradicts our assumption on $i_1$. Therefore, we have $i_2\in \{N_1+1, \cdots, N\}$ and $\lambda_{i_2}\big{|}_{G_2}$ is different from $\lambda_{l+1}\big{|}_{G_2}, \cdots, \lambda_{N_1}\big{|}_{G_2}$. 
\smallskip

Continuing this process, we have a sequence $i_1, i_2, \cdots$, with each item an element in $\{N_1+1, \cdots, N\}$. Take the minimal integer $p$ such that $i_{p+1}=i_q$ for certain $q\in\{1, \cdots, p\}$. Then the Klein polynomial $F_{s+1}=x_{i_q}^{d-1}x_{i_{q+1}}+\cdots+x_{i_{p-1}}^{d-1}x_{i_p}+x_{i_p}^{d-1}x_{i_q}$ satisfies that $S(F_{s+1})\subset S(F)$. This contradicts the maximality of $N_1$. Thus we conclude that the action of $G_2$ on $\PP_{x_{l+1}, \cdots, x_{N_1}}^{N_1-l-1}$ is faithful.
\smallskip

Next we construct $F_0$ as required in the theorem. If $k=l+1$, we take $N_0=N_1$ and 
\begin{equation*}
F_0(x_1, \cdots, x_{N_1})=x_1^{d-1}x_2+\cdots+x_{k-1}^{d-1}x_k+x_k^d+F_1+\cdots+F_s
\end{equation*}
which is a simple polynomial with exactly one component of type $T_{\ge 2}$. The action of $G$ on $\PP^{N-1}_{x_1, \cdots, x_N}$ then induces a faithful action of $G$ on $V(F_0)\subset \PP^{N_1-1}_{x_1,\cdots, x_{N_1}}$.

\smallskip

Now we assume that $k>l+1$. We aim to show that $N_1<N$. By Lemma \ref{lemma: criterion}, there exists a monomial $m\in S$ with the total degree of $x_l, x_k$ at least $d-1$, and the degree of $x_{l+1}$ being zero. We write $m=x_l^{u} x_k^{d-u-1} x_t$ with $t\ne l+1$. From $\lambda_l\big{|}_{G_2}=\lambda_k\big{|}_{G_2}$ and $\lambda_k^{d-1}\lambda_{l+1}=\lambda_l^u \lambda_k^{d-u-1}\lambda_t=\lambda$ we obtain $\lambda_t\big{|}_{G_2}=\lambda_{l+1}\big{|}_{G_2}$. Since $\lambda_{l+1}, \cdots, \lambda_k$ are distinct and they all take value $1$ at $g_0$, we know that $\lambda_{l+1}\big{|}_{G_2}, \cdots, \lambda_k\big{|}_{G_2}$ are distinct. Therefore, $t\notin \{l+1, \cdots, k\}$. From $k>l+1$ we have $\lambda_l\big{|}_{G_2}\ne \lambda_{l+1}\big{|}_{G_2}$, hence $t\ne l$. From the construction of $F_1, \cdots, F_s$, we know that $\lambda_t\big{|}_{G_2}=\lambda_{l+1}\big{|}_{G_2}$ is different from $\lambda_{k+1}\big{|}_{G_2}, \cdots, \lambda_{N_1}\big{|}_{G_2}$. Thus $t\notin\{k+1, \cdots, N_1\}$. Take values at $g_0\in G_1$ on both sides of $\lambda_l^u \lambda_k^{d-u-1}\lambda_t=\lambda$, we have that $\lambda_t(g_0)=\lambda_l(g_0)^{-u}$. This implies that $t\notin\{1,2,\cdots, l-1\}$. We conclude that $t>N_1$. In particular, we must have $N_1<N$.
\smallskip

Now let $N_0=N_1+1$. We take a simple polynomial 
\begin{equation*}
F_0(y_1, \cdots, y_{l+1}, x_{l+1}, \cdots, x_{N_1})=F_T(y_1, \cdots, y_{l+1})+F_K(x_{l+1}, \cdots, x_{N_1})
\end{equation*} 
with $F_T=y_1^{d-1}y_2+\cdots+y_l^{d-1}y_{l+1}+y_{l+1}^d$ and $F_K=F_1+\cdots+F_s$. The polynomial $F_0$ has $N_0=N_1+1\le N$ variables. The group $G_1$ is isomorphic to a subgroup of $G_{F_T}$, and $G_2$ is isomorphic to a subgroup of $G_{F_K}$. Combining with Proposition \ref{proposition: simple} we conclude that $G$ is isomorphic to a subgroup of $G_{F_0}$. 
\end{proof}

\begin{rmk}
\label{remark: cubic fourfold}
For our application to cubic fourfolds, we give a deeper analysis about the case $d=3$. In this case, we claim that if $k>l+1$, then $N_1\le N-2$ (thus $N_0=N_1+1$ can be chosen to be strictly smaller than $N$).

\smallskip

Assume to the contrary that $N_1=N-1$. Suppose $k=l+2$. Then $\lambda_{l+1}^2 \lambda_k=\lambda_k^2\lambda_{l+1}=\lambda$, which implies that $\lambda_k=\lambda_{l+1}$. This contradicts the minimality of $k$. Thus we have that $k\ge l+3$. We already find $t\in\{N_1+1, \cdots, N\}$ with $\lambda_t\big{|}_{G_2}=\lambda_{l+1}\big{|}_{G_2}$. The only possibility is that $t=N$. By Lemma \ref{lemma: criterion}, there exists a monomial $m\in S$ with the total degree of $x_N, x_{l+1}$ at least $d-1$, and the degree of $x_{l+2}$ zero. We write $m=x_N^v x_{l+1}^{d-v-1} x_{t'}$ with $t'\ne l+2$. By taking restrictions to $G_2$, we have $\lambda_{t'}\big{|}_{G_2}=\lambda_{l+2}\big{|}_{G_2}$, which implies that $t'\notin \{l, l+1, \cdots, N\}$. We also have $\lambda_{t'}(g_0)=\lambda_t(g_0)^{-v}=\lambda_{l}(g_0)^{uv}$, which implies that $t'\notin\{1,2,\cdots, l-1\}$. Which is a contradiction.
\end{rmk}

\begin{rmk}
\label{remark: lift}
In Theorems \ref{theorem: abelian} and \ref{theorem: main} we require the actions to be liftable. This condition actually holds in most situations. See Proposition \ref{proposition: lift} below for the case of abelian automorphism groups as we consider. There are more general results for liftability of (non-necessarily abelian) automorphism groups of smooth hypersurfaces, see \cite{oguiso2019quintic} (Theorem 4.8) and the very recent work \cite{gonzalez2020lift} by Gonz\'{a}lez-Aguilera, Liendo and Montero.
\end{rmk}

\begin{prop}
\label{proposition: lift}
Given integers $d\ge 3$ and $N\ge 3$. Let $G$ be an abelian finite group acting faithfully and linearly on a smooth $(N-2)$-fold of degree $d$. Suppose that there exists an element $g\in G$ such that the order of $G/\langle g\rangle$ is coprime to $\gcd(d,N)$. Then the action is liftable.
\end{prop}
\begin{proof}
We can take a decomposition $G=G_1\oplus\cdots\oplus G_k$ with each $G_i$ ($i=1, \cdots, k$) cyclic, such that the orders $|G_2|, \cdots, |G_k|$ are all coprime to $(d, N)$. Denote by $n_i$ the order of $G_i$. For each $i=1, \cdots, k$, we take $g_i$ to be a generator of $G_i$ and take its preimage $\widetilde{g_i}$ in $\GL(N, \CC)$ with order $n_i$. 
\smallskip

For any $\{i, j\}\subset\{1, \cdots, k\}$, we have $\widetilde{g_i}\widetilde{g_j}=\rho \widetilde{g_j}\widetilde{g_i}$ for a complex number $\rho\ne 0$. Taking determinant of $\widetilde{g_i}\widetilde{g_j}$ and $\rho\widetilde{g_j}\widetilde{g_i}$, we have that $\rho^N=1$. By looking at the action of $\widetilde{g_i}\widetilde{g_j}$ and $\rho\widetilde{g_j}\widetilde{g_i}$ on the polynomial defining the invariant $(N-2)$-fold of degree $d$, we have that $\rho^d=1$. We have $\widetilde{g_j}^{-1} \widetilde{g_i}\widetilde{g_j}=\rho\widetilde{g_i}$. Taking $n_i$-th power of both side, we have $\rho^{n_i}=1$. Similarly we have $\rho^{n_j}=1$. Since one of $n_i, n_j$ is coprime to $(d, N)$, we must have $\rho=1$.
\smallskip

From the above argument we conclude that each two of $\widetilde{g_1}, \cdots, \widetilde{g_k}$ commute. These $k$ elements then generate an abelian subgroup of $\GL(N, \CC)$ which is a lifting of $G$. 
\end{proof}

Next we apply Theorem \ref{theorem: main} to determine all possible orders of linear automorphisms of smooth hypersurfaces:
\begin{thm}
\label{theorem: order}
Given integers $d\ge 3$ and $N\ge 3$. Suppose $n$ is the order of a linear automorphism of a smooth $(N-2)$-fold of degree $d$, then $n$ is also the order of a linear automorphism of certain $(N-2)$-fold of degree $d$ defined by a simple polynomial $F$ with at most one component of type $T_{\ge 2}$. Explicitly, $n$ is either a factor of 
\begin{enumerate}[(i)]
\item $\frac{|1-(1-d)^N|}{d}$, or
\item $(d-1)^{N-1}$, or
\item $|1-(1-d)^a|$, where $1\le a\le N-1$, or
\item $\mathrm{lcm}(|1-(1-d)^{a_1}|, \cdots, |1-(1-d)^{a_t}|)$, where $t\ge 2, 1\le a_1<\cdots< a_t$ and $\sum_{i=1}^t a_i\le N$, or
\item $\mathrm{lcm}(|1-(1-d)^{a_1}|, \cdots, |1-(1-d)^{a_t}|, (d-1)^{b-1})$, where $t\ge 1, 1\le a_1<\cdots< a_t, b\ge 2$,  and $\sum_{i=1}^t a_i+b\le N$. 
\end{enumerate}
\end{thm}

\begin{proof}
Let $g$ be a linear automorphism of a smooth hypersurface $V(F)$ of degree $d$ in $\PP^{N-1}$. First we assume $\ord(g)$ is not coprime to $d-1$. From Theorem \ref{theorem: main}, the group $\langle g\rangle$ is isomorphic to a subgroup of $G_F$ for $F=F(x_1, \cdots, x_N)$ a simple polynomial of degree $d$ and with exactly one component of type $T_{\ge 2}$ (assume it to be $T_b$ with $b\ge 2$). Thus the order $\ord(g)$ of $g$ is equal to the order of a diagonal automorphism of $V(F)$. From Proposition \ref{proposition: simple} and Lemma \ref{lemma: order}, the number $\ord(g)$ belong to case $(ii)$ (when $b=N$) or case $(v)$ (when $b<N$).

\smallskip

Suppose $(\ord(g), d-1)=1$. By Theorem \ref{theorem: abelian}, $\ord(g)$ is equal to the order of a linear automorphism of an $(N-2)$-fold of degree $d$ which is defined by a $K$-pure simple polynomial $F$. Then by Proposition \ref{proposition: simple} and Lemma \ref{lemma: order}, the number $\ord(g)$ belong to case $(i)$ (when $F$ is Klein) or cases $(iii), (iv)$ (when $F$ is not Klein). The theorem follows.
\end{proof}

\begin{rmk}
When $d-1$ is a prime power and $G$ is a cyclic group, the two conditions $(a), (b)$ in Theorem \ref{theorem: main} automatically hold. In this case, Theorem \ref{theorem: order} is a corollary of Theorem \ref{theorem: main}  combining with Proposition \ref{proposition: simple}. However, in Theorem \ref{theorem: order} we do not make any assumptions on the degree $d$.
\end{rmk}

\section{Application to Cubic Fourfolds}
\label{section: cubic}
We end up with an application to cubic fourfolds. Due to its close relation to hyper-K\"ahler geometry and rationality problems, cubic fourfold has appealed a lot of interest recently. In \cite{gonzalez}, Gonz\'{a}lez-Aguilera and Liendo show that orders of automorphisms of smooth cubic fourfolds only have prime factors $2,3,5,7, 11$. An automorphism $g$ of a smooth cubic fourfold $X$ is called symplectic if the induced action on $H^{3,1}(X)\cong \CC$ is trivial. Lie Fu \cite{fu2016classification} classified all symplectic automorphisms of smooth cubic fourfolds whose orders are prime-powers. In \cite{AutCubics}, Laza and the author classified symplectic automorphism groups for smooth cubic fourfolds. In particular, we know that all possibilities for the orders of symplectic automorphisms are $1$, $2$, $3$, $4$, $5$, $6$, $7$, $8$, $9$, $11$, $12$, $15$. In next proposition we give all possible orders for automorphisms of smooth cubic fourfolds.

\begin{prop}
\label{proposition: cubic}
Let $n$ be the order of an automorphism of a smooth cubic fourfold. Then $n$ is a factor of $21, 30, 32, 33, 36$ or $48$. Each of those six integers can be realized as the order of an automorphism of a smooth cubic fourfold that is unique up to isomorphism. Moreover, we exhibit the constructions explicitly in Table \ref{table: cubic fourfold}.
\end{prop}

\begin{proof}
Firstly, by \cite{matsumura1963automorphisms}, any automorphism of a smooth cubic fourfold is linear and of finite order. We apply Theorem \ref{theorem: order} to the case $(d, N)=(3, 6)$. In case $(i)$, $n$ is a factor of $\frac{|1-(1-d)^6|}{d}=21$. In case $(ii)$, $n$ is a factor of $(d-1)^5=32$. The function $|1-(1-d)^a|$ takes values $3,3,9,15,33$ for $a=1,2,3,4,5$ respectively. Thus in case $(iii)$, $n$ can be a factor of $9, 15$ or $33$.
\smallskip

In case $(iv)$, there are choices $(a_1, \cdots, a_t)=(1,5), (2,4), (1,4), (2,3), (1,3),(1,2)$ or $(1,2,3)$. Straightforward calculations show that $\lcm(|1-(1-d)^{a_1}|,\cdots, |1-(1-d)^{a_t}|)=3,9$, $15$ or $33$. Thus in this case $n$ is a factor of $9$, $15$ or $33$.
\smallskip

Finally, in case $(v)$ we have choices $(a_1, \cdots, a_t, b)=(1,5)$, $(2,4)$, $(1,4)$, $(3,3)$, $(2,3)$, $(1,3)$, $(1,2,3)$, $(4,2)$, $(3,2)$, $(2,2)$, $(1,2)$, $(1,3,2)$, $(1,2,2)$. Straightforward calculations show that $\lcm(|1-(1-d)^{a_1}|,\cdots, |1-(1-d)^{a_t}|, (d-1)^{b-1})=6, 12, 18, 24, 30, 36$ or $48$. Thus the order $n$ is a factor of $30,36$ or $48$. We conclude that $n$ must be a factor of $21, 30, 32, 33, 36$ or $48$. Those six numbers are realized as orders of automorphisms of simple polynomials in the following table:
\begin{table}[ht] \caption{Here we denote by $\frac{1}{n}(\sigma_1, \cdots, \sigma_6)$ the automorphism of $\CC^6$ sending $(x_1, \cdots, x_6)$ to $(\zeta_n^{\sigma_1}x_1, \cdots, \zeta_n^{\sigma_6} x_6)$. Here $\zeta_n=\mathrm{exp}(\frac{2\pi\sqrt{-1}}{n})$.}
\label{table: cubic fourfold}
\renewcommand{\arraystretch}{1.2}\centering
\begin{tabular}{|c|c|c|c|}
\hline
Order  & Cubic Polynomial   & Maximal Automorphism  \\ [0.5ex]
\hline
$21$& $x_1^2 x_2+ x_2^2 x_3+ x_3^2 x_4+ x_4^2 x_5+ x_5^2 x_6+ x_6^2 x_1$ & $\frac{1}{63}(1, -2, 4, -8, 16, -32)$\\\hline
$30$& $x_1^2 x_2+ x_2^3 + x_3^2 x_4+ x_4^2 x_5+x_5^2 x_6+x_6^2 x_3$ & $\frac{1}{30}(15, 0, 2, -4, 8, -16)$ \\\hline
$32$& $x_1^2 x_2+ x_2^2 x_3+ x_3^2 x_4+ x_4^2 x_5+ x_5^2 x_6+ x_6^3$ & $\frac{1}{32}(1, -2, 4, -8, 16, 0)$   \\\hline
$33$& $x_1^3+ x_2^2 x_3+x_3^2 x_4+x_4^2 x_5+x_5^2 x_6+x_6^2 x_2$ & $\frac{1}{33}(11, 3, -6, 12, 9, -18)$  \\\hline
$36$& $x_1^2 x_2+ x_2^2 x_3+ x_3^3+ x_4^2 x_5+ x_5^2 x_6+ x_6^2 x_4$ & $\frac{1}{36}(9, -18, 0, 4, -8, 16)$ \\\hline
$48$& $x_1^2 x_2+ x_2^2 x_3+ x_3^2 x_4+ x_4^2 x_5+ x_5^3+ x_6^3$ & $\frac{1}{48}(3, -6, 12, -24, 0, 16)$     \\\hline
\end{tabular}
\end{table}
\smallskip

Given $n\in \{21, 30, 32, 33, 36, 48\}$, we denote by $F_1(x_1, \cdots, x_6)$ the simple polynomial in Table \ref{table: cubic fourfold} corresponding to $n$. Next we show the uniqueness of cubic fourfolds that admitting an automorphism of order $n$. Take $F(x_1, \cdots, x_6)$ to be a smooth cubic polynomial such that $V(F)$ admits a diagonal automorphism of order $n$. We aim to show that $V(F)$ and $V(F_1)$ are isomorphic to each other. Applying Theorem \ref{theorem: order} for $(d,N)=(3, 5)$, we know that the order of an automorphism of a smooth cubic threefold is a factor of $11,15,16,18$ or $24$. Thus the order of each automorphism of a smooth cubic threefold does not equal to $n\in \{21, 30, 32, 33, 36, 48\}$. Therefore, when we apply Theorem \ref{theorem: main} for $(d,N)=(3,6)$, the number $N_0$ must be equal to $6$. By Remark \ref{remark: cubic fourfold}, we have $k=l+1$ in the proof of Theorem \ref{theorem: main}. This allows us to make the choice of $F_0$ such that the two cubic fourfolds $V(F), V(F_0)\subset \PP^5_{x_1, \cdots, x_6}$ are preserved by the same automorphism $g\in \PGL(N, \CC)$ of order $n$.

\smallskip

Now we have two simple cubic polynomials $F_0, F_1$ in variables $x_1, \cdots, x_6$, such that both $G_{F_0}$ and $G_{F_1}$ contain elements with order $n$. From our calculation of the maximal numbers $21,30,32,33,36,48$, simple cubic polynomials in $(x_1, \cdots, x_6)$ with diagonal automorphisms of order $n$ are unique up to permutation of coordinates. Therefore, we can assume that $S(F_0)=S(F_1)$.
\smallskip

The types of $F_1$ for $n=21, 30, 32, 33, 36, 48$ are $K_6, T_2\oplus K_4, T_6, K_1\oplus K_5, T_3\oplus K_3, T_5\oplus K_1$ respectively. By Proposition \ref{proposition: simple}, the group $G_{F_1}$ is isomorphic to $\ZZ/21$, $\ZZ/30$, $\ZZ/32$, $\ZZ/33$, $\ZZ/36$, $\ZZ/48$ respectively. Therefore, diagonal automorphism of $V(F_1)$ of order $n$ is unique up to conjugation. Without loss of generality, we assume that the action of $g$ on $V(F_0)$ coincides with the explicitly described one in Table \ref{table: cubic fourfold}. We claim that $S(F)=S(F_0)$. The proof is straightforward and we take $n=48$ for an example. Suppose to the contrary that there exists a monomial $m\in S(F)-S(F_0)$. If the degree of certain $x_i$ is at least $2$ in $m$, then there is a unique $x_j$ with $g(x_i^2 x_j)=x_i^2 x_j$. Thus $m$ must belong to $S(F_0)$, which is a contradiction. Now assume $m=x_{i_1}x_{i_2}x_{i_3}$ with $i_1<i_2<i_3$. However, $g=\frac{1}{48}(3, -6, 12, -24, 0, 16)$ and the sum of any three distinct numbers in $\{3,-6,12,-24,0,16\}$ is not zero. This is a contradiction. 
\smallskip

Now we have $S(F)=S(F_1)$. By taking scalars on $x_1, \cdots, x_N$, the two polynomials $F$ and $F_1$ can be transformed to each other. Thus $V(F)$ is isomorphic to $V(F_1)$. We conclude the uniqueness.
\end{proof}

\bibliography{reference} 

\end{document}